\documentclass[11pt,leqno]{amsart}
\usepackage{amssymb}
\usepackage{xypic}
\begin{document}
\theoremstyle{plain}
\newtheorem{thm}{Theorem}[section]
\newtheorem*{thm1}{Theorem 1}
\newtheorem*{thm2}{Theorem 2}
\newtheorem{lemma}[thm]{Lemma}
\newtheorem{lem}[thm]{Lemma}
\newtheorem{cor}[thm]{Corollary}
\newtheorem{prop}[thm]{Proposition}
\newtheorem{propose}[thm]{Proposition}
\newtheorem{variant}[thm]{Variant}
\theoremstyle{definition}
\newtheorem{notations}[thm]{Notations}
\newtheorem{rem}[thm]{Remark}
\newtheorem{rmk}[thm]{Remark}
\newtheorem{rmks}[thm]{Remarks}
\newtheorem{defn}[thm]{Definition}
\newtheorem{ex}[thm]{Example}
\newtheorem{claim}[thm]{Claim}
\newtheorem{ass}[thm]{Assumption}
\numberwithin{equation}{section}
\newcounter{elno}                
\def\points{\list
{\hss\llap{\upshape{(\roman{elno})}}}{\usecounter{elno}}} 
\let\endpoints=\endlist


\catcode`\@=11
%
%
\def\opn#1#2{\def#1{\mathop{\kern0pt\fam0#2}\nolimits}} 
\def\bold#1{{\bf #1}}%
\def\underrightarrow{\mathpalette\underrightarrow@}
\def\underrightarrow@#1#2{\vtop{\ialign{$##$\cr
 \hfil#1#2\hfil\cr\noalign{\nointerlineskip}%
 #1{-}\mkern-6mu\cleaders\hbox{$#1\mkern-2mu{-}\mkern-2mu$}\hfill
 \mkern-6mu{\to}\cr}}}
\let\underarrow\underrightarrow
\def\underleftarrow{\mathpalette\underleftarrow@}
\def\underleftarrow@#1#2{\vtop{\ialign{$##$\cr
 \hfil#1#2\hfil\cr\noalign{\nointerlineskip}#1{\leftarrow}\mkern-6mu
 \cleaders\hbox{$#1\mkern-2mu{-}\mkern-2mu$}\hfill
 \mkern-6mu{-}\cr}}}
%
%

%
\def\:{\colon}
\let\oldtilde=\tilde
\def\tilde#1{\mathchoice{\widetilde{#1}}{\widetilde{#1}}%
{\indextil{#1}}{\oldtilde{#1}}}
\def\indextil#1{\lower2pt\hbox{$\textstyle{\oldtilde{\raise2pt%
\hbox{$\scriptstyle{#1}$}}}$}}
\def\pnt{{\raise1.1pt\hbox{$\textstyle.$}}}
%

%
\let\amp@rs@nd@\relax
\newdimen\ex@\ex@.2326ex
\newdimen\bigaw@l
\newdimen\minaw@
\minaw@16.08739\ex@
\newdimen\minCDaw@
\minCDaw@2.5pc
\newif\ifCD@
\def\minCDarrowwidth#1{\minCDaw@#1}
\newenvironment{CD}{\@CD}{\@endCD}
\def\@CD{\def\A##1A##2A{\llap{$\vcenter{\hbox
 {$\scriptstyle##1$}}$}\Big\uparrow\rlap{$\vcenter{\hbox{%
$\scriptstyle##2$}}$}&&}%
\def\V##1V##2V{\llap{$\vcenter{\hbox
 {$\scriptstyle##1$}}$}\Big\downarrow\rlap{$\vcenter{\hbox{%
$\scriptstyle##2$}}$}&&}%
\def\={&\hskip.5em\mathrel
 {\vbox{\hrule width\minCDaw@\vskip3\ex@\hrule width
 \minCDaw@}}\hskip.5em&}%
\def\verteq{\Big\Vert&&}%
\def\noarr{&&}%
\def\vspace##1{\noalign{\vskip##1\relax}}\relax\let\amp@rs@nd@&\iffalse}\fi
 \CD@true\vcenter\bgroup\relax\let\\=\cr\iffalse}\fi\tabskip\z@skip\baselineskip20\ex@
 \lineskip3\ex@\lineskiplimit3\ex@\halign\bgroup
 &\hfill$\m@th##$\hfill\cr}
\def\@endCD{\cr\egroup\egroup}
%
\def\>#1>#2>{\amp@rs@nd@\setbox\z@\hbox{$\scriptstyle
 \;{#1}\;\;$}\setbox\@ne\hbox{$\scriptstyle\;{#2}\;\;$}\setbox\tw@
 \hbox{$#2$}\ifCD@
 \global\bigaw@\minCDaw@\else\global\bigaw@\minaw@\fi
 \ifdim\wd\z@>\bigaw@\global\bigaw@\wd\z@\fi
 \ifdim\wd\@ne>\bigaw@\global\bigaw@\wd\@ne\fi
 \ifCD@\hskip.5em\fi
 \ifdim\wd\tw@>\z@
 \mathrel{\mathop{\hbox to\bigaw@{\rightarrowfill}}\limits^{#1}_{#2}}\else
 \mathrel{\mathop{\hbox to\bigaw@{\rightarrowfill}}\limits^{#1}}\fi
 \ifCD@\hskip.5em\fi\amp@rs@nd@}
\def\<#1<#2<{\amp@rs@nd@\setbox\z@\hbox{$\scriptstyle
 \;\;{#1}\;$}\setbox\@ne\hbox{$\scriptstyle\;\;{#2}\;$}\setbox\tw@
 \hbox{$#2$}\ifCD@
 \global\bigaw@\minCDaw@\else\global\bigaw@\minaw@\fi
 \ifdim\wd\z@>\bigaw@\global\bigaw@\wd\z@\fi
 \ifdim\wd\@ne>\bigaw@\global\bigaw@\wd\@ne\fi
 \ifCD@\hskip.5em\fi
 \ifdim\wd\tw@>\z@
 \mathrel{\mathop{\hbox to\bigaw@{\leftarrowfill}}\limits^{#1}_{#2}}\else
 \mathrel{\mathop{\hbox to\bigaw@{\leftarrowfill}}\limits^{#1}}\fi
 \ifCD@\hskip.5em\fi\amp@rs@nd@}
%
%
\newenvironment{CDS}{\@CDS}{\@endCDS}
\def\@CDS{\def\A##1A##2A{\llap{$\vcenter{\hbox
 {$\scriptstyle##1$}}$}\Big\uparrow\rlap{$\vcenter{\hbox{%
$\scriptstyle##2$}}$}&}%
\def\V##1V##2V{\llap{$\vcenter{\hbox
 {$\scriptstyle##1$}}$}\Big\downarrow\rlap{$\vcenter{\hbox{%
$\scriptstyle##2$}}$}&}%
\def\={&\hskip.5em\mathrel
 {\vbox{\hrule width\minCDaw@\vskip3\ex@\hrule width
 \minCDaw@}}\hskip.5em&}
\def\verteq{\Big\Vert&}
\def\novarr{&}
\def\noharr{&&}
\def\SE##1E##2E{\slantedarrow(0,18)(4,-3){##1}{##2}&}
\def\SW##1W##2W{\slantedarrow(24,18)(-4,-3){##1}{##2}&}
\def\NE##1E##2E{\slantedarrow(0,0)(4,3){##1}{##2}&}
\def\NW##1W##2W{\slantedarrow(24,0)(-4,3){##1}{##2}&}
\def\slantedarrow(##1)(##2)##3##4{%
\thinlines\unitlength1pt\lower 6.5pt\hbox{\begin{picture}(24,18)%
\put(##1){\vector(##2){24}}%
\put(0,8){$\scriptstyle##3$}%
\put(20,8){$\scriptstyle##4$}%
\end{picture}}}
\def\vspace##1{\noalign{\vskip##1\relax}}\relax\let\amp@rs@nd@&\iffalse}\fi
 \CD@true\vcenter\bgroup\relax\let\\=\cr\iffalse}\fi\tabskip\z@skip\baselineskip20\ex@
 \lineskip3\ex@\lineskiplimit3\ex@\halign\bgroup
 &\hfill$\m@th##$\hfill\cr}
\def\@endCDS{\cr\egroup\egroup}
%
\newdimen\TriCDarrw@
\newif\ifTriV@
\newenvironment{TriCDV}{\@TriCDV}{\@endTriCD}
\newenvironment{TriCDA}{\@TriCDA}{\@endTriCD}
\def\@TriCDV{\TriV@true\def\TriCDpos@{6}\@TriCD}
\def\@TriCDA{\TriV@false\def\TriCDpos@{10}\@TriCD}
\def\@TriCD#1#2#3#4#5#6{%
\setbox0\hbox{$\ifTriV@#6\else#1\fi$}
\TriCDarrw@=\wd0 \advance\TriCDarrw@ 24pt
\advance\TriCDarrw@ -1em
\def\SE##1E##2E{\slantedarrow(0,18)(2,-3){##1}{##2}&}
\def\SW##1W##2W{\slantedarrow(12,18)(-2,-3){##1}{##2}&}
\def\NE##1E##2E{\slantedarrow(0,0)(2,3){##1}{##2}&}
\def\NW##1W##2W{\slantedarrow(12,0)(-2,3){##1}{##2}&}
\def\slantedarrow(##1)(##2)##3##4{\thinlines\unitlength1pt
\lower 6.5pt\hbox{\begin{picture}(12,18)%
\put(##1){\vector(##2){12}}%
\put(-4,\TriCDpos@){$\scriptstyle##3$}%
\put(12,\TriCDpos@){$\scriptstyle##4$}%
\end{picture}}}
\def\={\mathrel {\vbox{\hrule
   width\TriCDarrw@\vskip3\ex@\hrule width
   \TriCDarrw@}}}
\def\>##1>>{\setbox\z@\hbox{$\scriptstyle
 \;{##1}\;\;$}\global\bigaw@\TriCDarrw@
 \ifdim\wd\z@>\bigaw@\global\bigaw@\wd\z@\fi
 \hskip.5em
 \mathrel{\mathop{\hbox to \TriCDarrw@
{\rightarrowfill}}\limits^{##1}}
 \hskip.5em}
\def\<##1<<{\setbox\z@\hbox{$\scriptstyle
 \;{##1}\;\;$}\global\bigaw@\TriCDarrw@
 \ifdim\wd\z@>\bigaw@\global\bigaw@\wd\z@\fi
 \mathrel{\mathop{\hbox to\bigaw@{\leftarrowfill}}\limits^{##1}}
 }
 \CD@true\vcenter\bgroup\relax\let\\=\cr\iffalse}\fi
 \tabskip\z@skip\baselineskip20\ex@
 \lineskip3\ex@\lineskiplimit3\ex@
 \ifTriV@
 \halign\bgroup
 &\hfill$\m@th##$\hfill\cr
#1&\multispan3\hfill$#2$\hfill&#3\\
&#4&#5\\
&&#6\cr\egroup%
\else
 \halign\bgroup
 &\hfill$\m@th##$\hfill\cr
&&#1\\%
&#2&#3\\
#4&\multispan3\hfill$#5$\hfill&#6\cr\egroup
\fi}
\def\@endTriCD{\egroup} 
\newcommand{\mc}{\mathcal} 
\newcommand{\mb}{\mathbb} 
\newcommand{\surj}{\twoheadrightarrow} 
\newcommand{\inj}{\hookrightarrow} \newcommand{\zar}{{\rm zar}} 
\newcommand{\an}{{\rm an}} \newcommand{\red}{{\rm red}} 
\newcommand{\Rank}{{\rm rk}} \newcommand{\codim}{{\rm codim}} 
\newcommand{\rank}{{\rm rank}} \newcommand{\Ker}{{\rm Ker \ }} 
\newcommand{\Pic}{{\rm Pic}} \newcommand{\Div}{{\rm Div}} 
\newcommand{\Hom}{{\rm Hom}} \newcommand{\im}{{\rm im}} 
\newcommand{\Spec}{{\rm Spec \,}} \newcommand{\Sing}{{\rm Sing}} 
\newcommand{\sing}{{\rm sing}} \newcommand{\reg}{{\rm reg}} 
\newcommand{\Char}{{\rm char}} \newcommand{\Tr}{{\rm Tr}} 
\newcommand{\Gal}{{\rm Gal}} \newcommand{\Min}{{\rm Min \ }} 
\newcommand{\Max}{{\rm Max \ }} \newcommand{\Alb}{{\rm Alb}\,} 
\newcommand{\GL}{{\rm GL}\,} 
\newcommand{\ie}{{\it i.e.\/},\ } \newcommand{\niso}{\not\cong} 
\newcommand{\nin}{\not\in} 
\newcommand{\soplus}[1]{\stackrel{#1}{\oplus}} 
\newcommand{\by}[1]{\stackrel{#1}{\rightarrow}} 
\newcommand{\longby}[1]{\stackrel{#1}{\longrightarrow}} 
\newcommand{\vlongby}[1]{\stackrel{#1}{\mbox{\large{$\longrightarrow$}}}} 
\newcommand{\ldownarrow}{\mbox{\Large{\Large{$\downarrow$}}}} 
\newcommand{\lsearrow}{\mbox{\Large{$\searrow$}}} 
\renewcommand{\d}{\stackrel{\mbox{\scriptsize{$\bullet$}}}{}} 
\newcommand{\dlog}{{\rm dlog}\,} 
\newcommand{\longto}{\longrightarrow} 
\newcommand{\vlongto}{\mbox{{\Large{$\longto$}}}} 
\newcommand{\limdir}[1]{{\displaystyle{\mathop{\rm lim}_{\buildrel\longrightarrow\over{#1}}}}\,} 
\newcommand{\liminv}[1]{{\displaystyle{\mathop{\rm lim}_{\buildrel\longleftarrow\over{#1}}}}\,} 
\newcommand{\norm}[1]{\mbox{$\parallel{#1}\parallel$}} 
\newcommand{\boxtensor}{{\Box\kern-9.03pt\raise1.42pt\hbox{$\times$}}} 
\newcommand{\into}{\hookrightarrow} \newcommand{\image}{{\rm image}\,} 
\newcommand{\Lie}{{\rm Lie}\,} 
\newcommand{\CM}{\rm CM}
\newcommand{\sext}{\mbox{${\mathcal E}xt\,$}} 
\newcommand{\shom}{\mbox{${\mathcal H}om\,$}} 
\newcommand{\coker}{{\rm coker}\,} 
\newcommand{\sm}{{\rm sm}} 
\newcommand{\tensor}{\otimes} 
\renewcommand{\iff}{\mbox{ $\Longleftrightarrow$ }} 
\newcommand{\supp}{{\rm supp}\,} 
\newcommand{\ext}[1]{\stackrel{#1}{\wedge}} 
\newcommand{\onto}{\mbox{$\,\>>>\hspace{-.5cm}\to\hspace{.15cm}$}} 
\newcommand{\propsubset} {\mbox{$\textstyle{ 
\subseteq_{\kern-5pt\raise-1pt\hbox{\mbox{\tiny{$/$}}}}}$}} 
\newcommand{\sB}{{\mathcal B}} \newcommand{\sC}{{\mathcal C}} 
\newcommand{\sD}{{\mathcal D}} \newcommand{\sE}{{\mathcal E}} 
\newcommand{\sF}{{\mathcal F}} \newcommand{\sG}{{\mathcal G}} 
\newcommand{\sH}{{\mathcal H}} \newcommand{\sI}{{\mathcal I}} 
\newcommand{\sJ}{{\mathcal J}} \newcommand{\sK}{{\mathcal K}} 
\newcommand{\sL}{{\mathcal L}} \newcommand{\sM}{{\mathcal M}} 
\newcommand{\sN}{{\mathcal N}} \newcommand{\sO}{{\mathcal O}} 
\newcommand{\sP}{{\mathcal P}} \newcommand{\sQ}{{\mathcal Q}} 
\newcommand{\sR}{{\mathcal R}} \newcommand{\sS}{{\mathcal S}} 
\newcommand{\sT}{{\mathcal T}} \newcommand{\sU}{{\mathcal U}} 
\newcommand{\sV}{{\mathcal V}} \newcommand{\sW}{{\mathcal W}} 
\newcommand{\sX}{{\mathcal X}} \newcommand{\sY}{{\mathcal Y}} 
\newcommand{\sZ}{{\mathcal Z}} \newcommand{\ccL}{\sL} 
 \newcommand{\A}{{\mathbb A}} \newcommand{\B}{{\mathbb 
B}} \newcommand{\C}{{\mathbb C}} \newcommand{\D}{{\mathbb D}} 
\newcommand{\E}{{\mathbb E}} \newcommand{\F}{{\mathbb F}} 
\newcommand{\G}{{\mathbb G}} \newcommand{\HH}{{\mathbb H}} 
\newcommand{\I}{{\mathbb I}} \newcommand{\J}{{\mathbb J}} 
\newcommand{\M}{{\mathbb M}} \newcommand{\N}{{\mathbb N}} 
\renewcommand{\P}{{\mathbb P}} \newcommand{\Q}{{\mathbb Q}} 

\newcommand{\R}{{\mathbb R}} \newcommand{\T}{{\mathbb T}} 
\newcommand{\U}{{\mathbb U}} \newcommand{\V}{{\mathbb V}} 
\newcommand{\W}{{\mathbb W}} \newcommand{\X}{{\mathbb X}} 
\newcommand{\Y}{{\mathbb Y}} \newcommand{\Z}{{\mathbb Z}} 
\title[HK multiplicity, $F$-threshold and the Paley-Wiener theorem] 
{HK multiplicity, $F$-threshold and the Paley-Wiener theorem} 
\author{V. Trivedi} 
\address{School of Mathematics, Tata Institute of 
Fundamental Research, Homi Bhabha Road, Mumbai-400005, India} 
\email{vija@math.tifr.res.in} 
\date{}

\maketitle
\begin{abstract}{
For  a given algebraically closed  field $k$ of characteristic $p>0$ we consider the set $\sC_k$, of 
graded isomorphism classes of {\em standard graded pairs} $(R, I)$, where 
$R$ is a standard graded ring over the field and $I$ is  a graded ideal of 
finite colength.

 Here we give   
a ring homomorphism 
 $\Pi:\Z[\sC_k] \longto  H(\C)[X]$, where  $H(\C)$ denotes the ring  of entire functions.

The related entire function and the homomorphism $\Pi$ keep track of the two 
 positive characteristic  invariants, $e_{HK}(R, I)$ and $c^I({\bf m})$  of the ring:
(1) composing the map $\Pi$ with the evaluation map at $z=0$ 
gives a ring homomorphism 
 $\Pi_e:\Z[\sC_k] \longto \R[X]$
which sends 
$$(R,I) \to e_{HK}(R^0, IR^0)+ e_{HK}(R^1, IR^1)X+\cdots + 
e_{HK}(R^d, IR^d)X^d,$$
where $R^i$ is the union of $i$ dimensional components 
of $R$ and $e_{HK}(R^i, IR^i)$ is the HK multiplicity of the pair $(R^i, IR^i)$, and in particular the top coefficient is $e_{HK}(R, I)$.

(2) If, in addition,  $R$ is a two dimensional ring or 
$\mbox {Proj~R}$ is strongly $F$-regular, then the   Fourier transform 
${\widehat f}_{R, I}$
 belongs to the Paley-Wiener class of the real number, namely 
 the $F$-threshold $c^I_{\bf m}(R)$ of the maximal ideal ${\bf m}$.}\end{abstract}

\section{Introduction}
Let $(R, I)$ be a standard graded pair, {\em i.e.}, $R$ is a Noetherian
standard graded ring   over an algebraically closed  field $k$ (unless otherwise stated) of
characteristic $p >0$
and $I$ is a graded ideal of finite colength.
 Let ${\bf m}$ be the graded maximal ideal of $R$.

To study Hilbert-Kunz (HK) multiplicity $e_{HK}(R, I)$ of $R$ with respect to $I$, 
we had introduced in [T] a compactly supported continuous function called HK density 
function $f_{R, I}:[0, \infty) 
\longrightarrow [0, \infty)$ provided $\dim R\geq 2$.
This function relates to the two characteristic $p$ invariants of the pair
$(R, I)$.
 \begin{enumerate}
\item $\int_0^{\infty} f_{R,I}(x)dx = e_{HK}(R, I)$.
\item If $R$ is strongly $F$-regular in the punctured spectrum and 
of dimension $\geq 2$, 
or a standard graded domain of dimension $= 2$ then the maximum support $\alpha(R, I)$ 
 of $f_{R, I}$ 
is same as the $F$-threshold of $c^I({\bf m})$.
\end{enumerate}

The function $f_{R, I}$ was given by
$$f_{R, I}(x) = \lim_{n\to \infty}f_n(R, I)(x), \quad\mbox{for all}\quad x\in \R,$$
where $\{f_n(R, I):[0, \infty)
\longrightarrow [0, \infty)\}_{n\in \N}$ is a 
sequence of compactly supported step functions which converges uniformly 
to $f_{R, I}$.

If $\dim~R = 0$ or $1$  then we can still define the functions
$\{f_n(R, I)\}_n$ exactly in the  same way and they 
 are compactly supported step functions.

If $\dim~R = 1$ then this sequence 
 converges pointwise everywhere except 
at finitely many points, and in fact
 uniformly outside a set of arbitrarily small measure.
Therefore again 
$$e_{HK}(R, I):= \lim_{n\to \infty}\int_0^\infty f_n(R, I)(x)dx = 
\int_0^{\infty} f_{R, I}(x)dx.$$

However, if $\dim~R = 0$ then the sequence $\{f_n(R, I)\}_n$ does converge 
pointwise but  
 $f_{R,I} := \lim_{n\to \infty}f_n(R, I)$ is  $0$ everywhere. Whereas 
$$e_{HK}(R, I)  = \ell(R) \neq  \int_0^\infty f_{R, I}(x)dx = 0.$$

In this paper to every standard graded pair $(R,I)$ of dimension $d\geq 0$,  
we associate  an
entire function (that means a holomorphic function on whole of $\C$) 
$F_{R, I}$ in a uniform manner.  The function 
$F_{R, I}$ keeps track of the invariants $e_{HK}(R, I)$ and 
$\alpha(R,I)$, and hence $c^I({\bf m})$ whenever it coincides with $\alpha(R, I)$.

To do this, we use the Fourier transform, as follows.

We note that  each $f_n(R, I)$ is a compactly supported step function 
  and therefore belongs to $L^1(\R)$.
Now the Plancherel Theorem (see preliminaries)  implies that the Fourier 
transform ${\widehat f}_n(R, I)$ of $f_n(R, I)$ is 
 a well defined entire function.

We show that 
$\lim_{n\to \infty} {\widehat f_n}(R, I)(z)$ exists for every $z\in \C$ and for 
every $d\geq 0$.
If we denote the limiting 
function as $F_{R, I}:\C\longrightarrow \C$, {\em i.e.}, 
 
$$F_{R, I}(z) := \lim_{n\to \infty} {\widehat f_n}(R, I)(z)$$
then $F_{R, I}(0) = e_{HK}(R, I)$.

If  $\dim~R\geq 1$ then  $F_{R, I}$  is the Fourier transform of the 
HK density function 
$f_{R, I} = \lim_{n\to \infty} f_n({R, I})$.  
Moreover if $\dim~R=1$ then there is a finite set $T_{R, I}$ of integers 
such that if $(S, J)$ is a standard graded pair then 
$$F_{R, I} = F_{S, J}\iff 
f_{R, I} = f_{S, J}\quad \mbox{for all}\quad x\in \R\setminus T_{R, I}.$$
Since $f_{R,I}$ is continuous if $\dim~R\geq 2$ we get
$\{f_{R,I}\mid \dim~R\geq 2\}\hookrightarrow H(\C)$.

One of the main result (Theorem~\ref{t1}) of this paper is to show that 
the correspondence  $(R, I) \to F_{R,I}$ is algebraic,  in the 
following sense.

Let $\sC_k$ denote the isomorphism classes of standard graded 
pairs $(R, I)$, where $R_0 = k$ is an algebraically closed  field. Then there is a multiplicative 
map of monoids 
$$\Phi:(\sC_k, \tensor) \longrightarrow H(\C)~~\text{given by}~~ (R, I) \to 
F_{R, I},$$
where the identity element $(k, (0))$ of the monoid $\sC_k$ maps to 
the identity element of $H(\C)$, namely the constant map 
$F_{k, (0)}:\C\longto \C$ given by $z \to 1$. 

Further  this map extends to the ring homomorphism 
$\Pi:\Z[\sC_k]\longrightarrow H(\C)[X]$ such that 
$$(R, I) \to F_{R^0, IR^0} + F_{R^1, IR^1}X + \cdots + 
F_{R^d, IR^d}X^d,$$
 where $R^i$ denotes the $i^{th}$-dimensional component of $R$ and $d = \dim~R$.

Moreover if $e:H(\C)[X] \longrightarrow \C$ denotes the evaluation map at $0$, 
{\em i.e.}, $F\to F(0)$ then the composition map
$e\circ \Pi:\Z[\sC_k] \longrightarrow \R[X]$ is a ring homomorphism
which sends
$$(R,I) \to e_{HK}(R^0, IR^0)+ e_{HK}(R^1, IR^1)X+\cdots +
e_{HK}(R^d, IR^d)X^d,$$
where  we know by the existing theory that  $e_{HK}(R^d, IR^d) = e_{HK}(R, I)$. 

If $\sC_k^{1} = \{(R, I)\in \sC_k\mid \dim~R\geq 1\}$
then $\sC_k^{1}$ is $\Z[\sC_k]$-module such that
$\Pi\mid_{\sC_k^1}:\sC_k^{1}\longrightarrow H(\C)[X]$ given by  
$$(R, I) \to {\widehat f}_{R^1, IR^1}X + \cdots + 
{\widehat f}_{R^d, IR^d}X^d$$
is  a $\Z[\sC_k]$-linear map and therefore 
$(e\circ \Pi)_{1}:\sC_k^{1} \longrightarrow \R[X]$ is a $\Z[\sC_k]$-linear
 which sends
$$(R,I) \to e_{HK}(R^1, IR^1)X+\cdots +
e_{HK}(R^d, IR^d)X^d.$$

Recall that an entire function $F:\C\longto \C$ is of {\em exponential} 
type, if there exist constants $c_0$ and $c_1$ such that 
$|F(z)|\leq c_0e^{c_1|z|}$, for all $z\in \C$. Here we say that the 
exponential index of $F$ is $(c_0, c_1)$ if $c_0$ and $c_1$ are the 
smallest real numbers with this property, and denote this by 
$\mbox{exp.index of}~F = (c_0, c_1)$. 

\begin{thm}\label{t2}For a standard graded pair $(R,I)$ the function  
${F_{R,I}}$ is an entire  function and 
$\mbox{exp.index of}~F = (e_{HK}(R, I), \alpha(R, I))$. That means
 for all $z\in \C$, 
 $$|{F_{R,I}}(z)|  \leq e_{HK}(R,I)e^{\alpha(R,I)|z|},\quad\mbox{where}\;\;
 \alpha(R, I) =\;\;\mbox{the maximal support of}\;\; f_{R,I}.$$

Moreover  $\alpha(R, I)$ is the smallest  real number such that  ${F_{R,I}}$
belongs to the Paley-Wiener class of $\alpha(R, I)$, {\em i.e.},
$${F_{R,I}}\in PW_{\alpha(R,I)}~\mbox{and}~  
{F_{R,I}} \not\in PW_A,~\mbox{ if}~ A < \alpha(R,I).$$

Further for $(R, I)$, $(S, J)\in \sC_k$
$$\begin{array}{lcl}
\mbox{exp.index  of}~F_{(R,I)\tensor (S,J)} & = &  
(\mbox{exp.index  of}~F_{R, I})(\mbox{exp.index  of}~F_{S, J})\\\
& = &  (e_{HK}(R, I)\cdot e_{HK}(S,J),~~ \alpha(R, I)+\alpha(S, J)).\end{array}$$
\end{thm}

The above theorem  gives the following 
\begin{cor}\label{c2}If $(R, I)$ is a standard graded pair and  either (1) 
 $R$ is 
a two dimensional domain or (2) $\mbox{Proj}~R$ is a strongly $F$-regular then 
the entire function  
${F_{R,I}}$ has exponential index 
$(e_{HK}(R, I), c^I({\bf m}))$ and $F_{R, I}\in PW_{c^I({\bf m})}$.
\end{cor}

\section{preliminaries}

\subsection{Some fundamental results  from analysis} In this section 
we recall basics from real and complex analysis. 
For further details reader may
refer to the classical book of W. Rudin ([R]).

Here we consider $\R$ as an Euclidean space and with Lebesgue measure $dt$.
By {\em almost everywhere} (a.e.)  we mean outside a measure $0$ subset of $\R$.

If $z = x+iy \in \C$, then  $|z| = \sqrt{x^2+y^2}$.

If $f:X\longto \C$ then the supremum norm of $f$ is given as 
$\|f\| = \mbox{sup}~\{|f(z)|\mid {z\in X}\}$.

 For an integer $1\leq p <\infty $ and a measurable 
space $X$, 
$$L^p(X) = \{\mbox{the measurable functions}\quad f:X\longto \C\quad \mid
\int_X|f(x)|^pdx < \infty\},$$ 
and the  $L^p$ norm on the space  $L^p(X)$ is given by 
$\|f\|_p = \int_X|f(x)|^pdx$.

\vspace{5pt}

\noindent Example: In this paper we would be dealing mainly with the following 
set  of functions.
$$C^{ae}_c(\R) = \{f:\R\longto \R\mid f~~\mbox{is compactly supported
 and a.e. continuous}\}.$$  
It follows easily that  $f\in C^{ae}_c(\R)$ implies 
 $f\in L^p(\R)$, for every $1\leq p < \infty $.

\vspace{5pt}

The binary operation 
$\ast: L^1(\R)\times L^1(\R) \longto L^1(\R)$ given by 
$$(f\ast g)(x) =  \int_{-\infty}^{\infty}f(x-t)g(t)dt,\quad\mbox{for}\quad
 f, g \in L^1(\R)$$ 
is called the 
 {\em convolution}.
It is easy to check that 
the map restricted to the space $C_c^{ae}(\R)$ 
 gives the map $\ast: C_c^{ae}(\R)\times C_c^{ae}(\R) \longto 
C_c(\R)$.
 In fact 
if  $\mbox{supp}~ f = [0, a_1]$ and 
$\mbox{supp}~ g = [0, a_2]$ then 
$f\ast g$ is a compactly supported continuous function with 
$\mbox{supp}~ f\ast g  = [0, a_1+a_2]$.

\vspace{5pt}

The set 
$$H(\C) = \{f:\C\longto \C\mid f\quad\mbox{is an entire function}\}$$
is a commutative ring with pointwise addition and multiplication.

The
{\em Fourier transform} map, $F:C_c^{ae}(\R) \longto H(\C)$ 
is given by $f\to {\widehat f}$, where 
$${\widehat f}(z) = \int_{\R} f(t)e^{i tz}dt, \quad\mbox{for}\quad z\in \C.$$

Moreover the map $F$ is multiplicative: it takes convolution of two 
functions to the pointwise multiplication of their Fourier transforms, {\em i.e.},
  ${\widehat {f\ast g}} = 
{\widehat f}\cdot {\widehat g}$ (see 9.2~(c)~~[R]).

\vspace{5pt}

A function $g:\R\longto \R$  is said to vanish at infinity, if  
for a given $\epsilon >0$ there is a compact set $K \subset \R$ such that 
$|g(x)|<\epsilon $, for all $x\not\in K$. 
 $$\begin{array}{lcl}
 C_0(\R) & = & \{g:\R\longto \R\mid g~\mbox{is 
continuous function and vanishes at infinity}\},\\\
C_c(\R)  & = & \{g:\R\longto \R\mid g~\mbox{is 
continuous function and has compact support}\}.\end{array}$$

\vspace{5pt}

The functions arising as   the HK density functions of standard 
graded pairs belong to the set  $C_c^{ae}(\R)$.
The relation between such functions and their  Fourier transforms works 
very well due to
the Plancherel theorem (page~404 in [R]) 
 and Paley-Wiener theorem (Theorem~19.2 in [R]).

\vspace{5pt}

\noindent{\bf The Plancherel Theorem}.\quad If $f\in L^1(-A, A)$, for 
some $A \in \R_+$ then its   
{\em  Fourier transform}  
is an entire function (denoted as ${\widehat f}\in H(\C))$
such that 
\begin{enumerate}
\item the  restriction of ${\widehat f}$ to the real axis lies in $L^2(\R)$. Moreover
\item  ${\widehat f}$ is of
exponential type {\em i.e.} there exist positive constants $C$ and $A$ such that
$$|{\widehat f}(z)| \leq C e^{A|z|}~~\mbox{for all}~~ 
z\in \C,~~\mbox{where}~~C = \int_{-A}^A|f(t)|dt.$$
\end{enumerate}

\vspace{5pt}

The Paley-Wiener theorem is 
 the  converse of the Plancherel Theorem.

\vspace{5pt}

\noindent{\bf The Paley-Wiener theorem}.\quad 
For an entire function $G:\C\longto \C$,
\begin{enumerate}
\item if $A$ and $C$ are positive constants such that   
$|G(z)| \leq C e^{A|z|}$, for all $z\in \C$ and 
\item $G\mid_{\R}\in L^2(\R)$
\end{enumerate}
then there exists a real valued function $f\in L^2([-A, A])$ such that 
$G(z) = {\widehat f}(z)$.
Such an entire function $G$ is said to {\em belong to the Paley-Wiener class} $A$ and this is 
denoted by $G \in PW_A$.

\vspace{5pt}

\begin{rmk}If $f$, $g \in C_c^{ae}(\R)$ then 
$${\widehat f} = {\widehat g} \implies \|{\widehat f}-{\widehat g}\|_2 = 0 \iff
\|f- g \|_2 = 0 \implies f =  g\quad \mbox{a.e.},$$
where the second implication follows from the fact (Theorem~9.13~(c) in [R])
that the map $F:L^2(\R)\longto L^2(\R)$ given by $f\to {\widehat f}$ is an 
isomorphism of Hilbert-space, {i.e.}, $\|f-g\|_2 = 
\|{\widehat f}-{\widehat g}\|_2$.
In particular,  if $f, g \in C_c(\R)$ then $f = g$. This gives an embedding 
  $F\mid_{C_c(\R)}:C_c(\R)\hookrightarrow  H(\C)$.
\end{rmk}
\vspace{5pt}

\subsection{Some relevant results from commutative algebra}

\begin{defn}
Let $k$ be  an algebraically closed field  
of characteristic $p>0$. 
Let $R = \oplus_{n\geq 0}R_n$ be a 
standard   graded Noetherian ring   such that 
$R_0 = k$. Let   ${\bf m} = \oplus_{m\geq 0} R_m$ denote  
the graded maximal ideal of $R$. Let 
 $I\subset R$ be  a  graded ideal of finite colength.

\vspace{5pt}

Henceforth we will called
such a pair  an {\em SG pair}.\end{defn}

Given a SG pair $(R, I)$ of dimension $d\geq 0$ we can associate a sequence of step functions 
$\{f_n(R, I):[0, \infty)\longto [0, \infty)\}_n$ which are given by

\begin{equation}\label{step}
f_n(R, I)(x) = \frac{1}{q^{d-1}}\ell\Big(\frac{R}{I^{[q]}}\Big)_{\lfloor xq\rfloor},
\quad \mbox{where}\quad q= p^n.\end{equation} 

This is a  compactly supported a.e. continuous function. In 
fact the support of $f_n(R, I) \subset [0, n_0]$, 
where  $n_0$ is a  number independent of $n$.

\vspace{5pt}

\noindent{\bf Theorem}\label{T}~(Theorem 1.1 of [T]).\quad If $R$ is a standard 
graded ring with $\dim~R\geq 2$ then 
$\{f_n(R, I)\}$ is a uniformly convergent  sequence. If $f_{R, I} = 
\lim_{n\to \infty}f_n(R, I)$ then  $f_{R, I}$ is a compactly sopported 
continuous function and 
$$e_{HK}(R,I) = \int_{\R}f_{R, I}(x)dx.$$

\vspace{5pt}

\begin{defn}For a given SG pair $(R, I)$ the function $f_{R,I}$ is the {\em HK 
density function} of $(R, I)$.
Let 
$$\alpha(R, I) = \mbox{sup}~\{x \mid f_{R, I}(x)\neq 0\}$$
the maximum support of $f_{R, I}$.
\end{defn}

The HK density function relates to another 
characteristic $p$ invariant of the ring.
We first  recall the notion (introduced by [MTW]) of 
$F$-threshold of $I$ with respect to $J$ 
$$c^J(I) := \lim_{e\to \infty}\frac{\max\{r\mid I^r\nsubseteq 
J^{[p^e]}\}}{p^e},$$
where the existence of this limit was proved in [DsNbP].

\vspace{5pt}

We also recall the following definition from [HH].

\begin{defn}\label{d6}A  Noetherian domain $R$ such that $R\longto R^{1/p}$
is module finite over $R$, is {\em strongly $F$-regular} if for every nonzero
$c\in R$ there exists $q$ such that $R$-linear map $R\longto R^{1/q}$
that sends $1$ to $c^{1/q}$ splits as a map of $R$-modules, {\em i.e.}
iff $Rc^{1/q}\subseteq R^{1/q}$ splits over $R$.
\end{defn}

Note that if $\mbox{Proj}~R$ is smooth then it is $F$-regular on the puctured 
spectrum.

\vspace{5pt}

\noindent{\bf Theorem}\label{TrW}~(Theorem~4.9, [Tr W]).\quad{\em Let $(R, I)$ be a standard 
graded pair
and ${\bf m}$ be the graded maximal ideal of $R$.
If $R$ is strongly $F$-regular on the punctured
spectrum (for example if $\mbox{Proj}~R$ is smooth) then
$\alpha(R, I) = c^I({\bf m}).$}

\vspace{5pt}

\noindent{\bf Theorem~}\label{T2}~(Theorem~C, [T2]).~~{\it Let $(R, I)$ be a
 standard graded pair where $R$ is a two dimensional domain. Then
$$c^I({\bf m}) = \alpha(R, I).$$}

\begin{rmk}If $\dim~R =1$ then it is obvios that 
$$c^I({\bf m}) = \lim_{e\to \infty}\mbox{max}\{x\mid 
\ell\Big(\frac{R}{I^{[p^e]}}\Big)_{\lfloor xq
\rfloor} \neq 0 \}= \alpha(R, I).$$
If $\dim~R = 0$ then again  $c^I({\bf m}) = 0$ and 
$$\alpha(R, I) = \mbox{Sup}~\{x\mid \lim_{n\to \infty}f_n(R, I)(x)\neq 0\} = 0.$$
\end{rmk}

\section{A map from SG pairs to entire functions}

We again recall that by a SG pair  we mean 
 $R$ is a standard graded ring over an algebraically closed field  $k$ 
and $I \subset R$ is a graded ideal such that $\ell(R/I) <\infty$.
In the rest of the paper, all  SG pairs $(R,I)$  are considered over 
a fixed algebraically closed field  
$k$ of characteristic $p>0$.

We consider a monoid $(\sC_k, \tensor)$  generated by sg pair as follows.
Let 
$$\sC_k= \{(R,I)\mid (R,I)~~\mbox{a sg pair over}~k\}\}/\equiv,$$ 
where $(R, I) \equiv (S,J)$ if there is a 
map $\eta:R\longrightarrow S$
which is a graded isomorphism of rings of degree $0$  
 such that $\eta(I) = J$. In other words
 $\sC_k$ is the  isomorphic classes of sg pairs.
For
$(R, I), (S, J)\in \sC_k$, we define
$$ (R, I)\tensor (S, J) = (R\tensor_k S, I\tensor_k J),~~\mbox{where}~~ 
(R\tensor S)_n = \oplus_i (R_i\tensor S_{n-i})$$
and $I\tensor_k J = \oplus_i (I_i\tensor_kJ_{n-i})$.

The identity element of this monoid is $(k, (0))$, where $k = 
k\oplus 0 \oplus \cdots $ is the standard graded ring with 
 $0^{th}$ component as $k$ and $n^{th}$ component $=0$ if $n\neq 0$.

Now for  a SG pair  $(R,I)$ of dimension $d$, we associate an entire function.
Consider $f_n(R, I)$ as in (\ref{step}).
Since each $f_n(R, I)$ is a compactly supported a.e. 
continuous function, its Fourier transform is an entire function.

We know (see [T]) that if 
$\dim~R\geq 2$ then   $\{f_n(R, I)\}_n$ converges uniformly to 
$f_{R,I}$, the HK density  function of $(R,I)$. Moreover if  $\dim~R =1$ and 
$R$ is reduced then the convergence is pointwise.

\begin{lemma}\label{l11}If $(R, I)$ is a SG pair and $\dim R = 1$ then
there exists a finite set of integers 
$T_{R, I} = \{0 = d_0 < d_1 < d_2 <\cdots < d_{s_1}\}$ and a set of constants
$C_1 > C_2 >  \cdots > C_{s_1} > 0 $ such that for all $p^n = q\gg 0$ we have 

$$\begin{array}{llll}
f_n(R, I)(x) & \leq  &  C_{i+1} &  
\mbox{if}~~~ \frac{\lfloor xq\rfloor}{q} \in [d_i,d_i+\frac{1}{q}, \ldots, 
d_{i}+\frac{m_0-1}{q})\\\
& = & C_{i+1} & \mbox{if}~~ \frac{\lfloor xq\rfloor}{q}\in 
[d_i+\frac{m_0}{q}, d_i+\frac{m_0+1}{q}, \ldots, d_{i+1} -\frac{1}{q})\\\
 & = & 0 & \mbox{if}~~~\frac{\lfloor xq\rfloor}{q} \geq d_{s_1}+\frac{m_0}{q}, 
\end{array}$$
where  $1\leq i+1 \leq {s_1}$. 

In particular 
$f_{R,I}(x) = C_{i+1}$ if $x\in (d_i, d_{i+1})$ and 
 the sequence $\{f_n(R, I)\}$ converges to 
$f_{R,I}$ uniformly outside a set 
of any arbitrarily small measure which contains $T_{R, I}$. 
\end{lemma}
\begin{proof}
Let $H^0_{\bf m}(R)  = Q_1\cap~ Q_2 \cap~ \cdots \cap~ Q_s$ be the 
primary decomposition of $H^0_{\bf m}(R)$ in $R$. Let $\{p_1, \ldots, p_s\}$ 
be the corresponding prime ideals. Then $\dim~R/Q_i =1$.

We note that, the canonical map 
$g:R \longto  R/Q_1\times \cdots \times R/Q_s$
has kernel and cokernel of finite length, as the induced map by localizing at 
$Q_i$   is an isomorphism for all $i$.

Note that $\mbox{Ass}~R/Q_i = \{\sqrt{Q_i} = p_i\}$
and there exists $m_0$ such that $\ell(R/Q_i)_m = e_0(R/Q_i, {\bf m})$,
 for $m\geq m_0$, and  $\ell(R/Q_i)_m \leq  e_0(R/Q_i, {\bf m})$ otherwise.

Moreover, for  any element $x$ of $I$, either $x$ is nilpotent or a non zerodivisor
in $R/Q_i$.
Hence we can choose $q \gg 0$ such that $(I^{[q]}+Q_i)/Q_i$ is generated by 
non zerodivisor of $R/Q_i$, for all $i$.
Let 

$$d_{Q_i} = \mbox{min} \{\deg x\mid x\in I~~\mbox{and is a non zerodivisor of}~~ R/Q_i\}.$$
We choose  $q\gg 0$ such that  $qd_{Q_i} > m_0$.
Now

$$\begin{array}{llll}
\ell(\frac{R}{I^{[q]}+Q_i})_m &  =  \ell(R/Q_i)_m  & \leq e_0(R/Q_i) 
& \mbox{if}~~m< m_0\\\
  &  =   \ell(R/Q_i)_m & =  e_0(R/Q_i) & \mbox{if}~~m_0 \leq  m < qd_{Q_i}\\\
    & <  \ell(R/Q_i)_m & =  e_0(R/Q_i) & \mbox{if}~~ qd_{Q_i} \leq  m < m_0+ qd_{Q_i}\\\
 & =  0 & & \mbox{if}~~  m_0+ qd_{Q_i} \leq m,
\end{array}$$
where the last assertion follows by choosing a non zerodivisor $x_i\in I$ 
of the ring $R/Q_i$ and 
considering
 the exact sequence 
$$0\longto \big(\frac{R}{Q_i}\big)_m\longby{\times x_i^q} 
\big(\frac{R}{Q_i}\big)_m \longto 
\big(\frac{R}{x_i^q+Q_i}\big)_m\longto 0.$$
Now we have 
$$\begin{array}{llll}
f_n\big(\frac{R}{Q_i}, \frac{I+Q_i}{Q_i}\big)(x) & \leq & e_0({R}/{Q_i}) & 
\mbox{if}~~ \frac{\lfloor xq\rfloor}{q}\in \{\frac{1}{q}, \ldots, \frac{(m_0-1)}{q}\}\\\
 & =  & e_0(R/Q_i) & 
\mbox{if}~~ \frac{\lfloor xq\rfloor}{q} \in \{\frac{m_0}{q}, \ldots, 
d_{Q_i}-\frac{1}{q}\}\\\
& < & e_0({R}/{Q_i}) & 
\mbox{if}~~ \frac{\lfloor xq\rfloor}{q}\in \{d_{Q_i}, 
\ldots, d_{Q_i}+\frac{(m_0-1)}{q}\}\\\
  & = & 0 & 
\mbox{if}~~ \frac{\lfloor xq\rfloor}{q}\geq d_{Q_i}+\frac{m_0}{q}\end{array}$$

We  partition the set
$\{Q_1, \ldots, Q_s\} = {\tilde Q_1} \cup  
{\tilde Q_2} \cup \cdots \cup {\tilde Q_{s_1}}$ such that
 $Q_{j_1}, Q_{j_2}\in {\tilde Q_i}$ implies 
$d_{Q_{j_1}} =  d_{Q_{j_2}} =: d_i$ 

 Without loss of generality we assume 
  $d_1 <  d_2 < \cdots < d_{s_1}$.

Now we further choose $q_0$ such that for $q\geq q_0$
$$m_0
< q d_1 < q d_1 + m_0 < qd_{2} < \cdots <
 qd_{i} < qd_{i}+m_0 < qd_{i+1} <\cdots < q d_{s_1}.$$
 
Hence 
$$\begin{array}{llll}
f_n(R, I)(x) & \leq & \sum_{Q\in {\tilde Q_1}\cup \cdots \cup {\tilde Q_{s_1}}}
e_0(R/Q) & 
\mbox{if}~~ \frac{\lfloor xq\rfloor}{q}\in \{\frac{1}{q},\ldots,  \frac{(m_0-1)}{q}\}\\\
 & =  & \sum_{Q\in {\tilde Q_1}\cup \cdots \cup {\tilde Q_{s_1}}}e_0(R/Q) & 
\mbox{if}~~ \frac{\lfloor xq\rfloor}{q}\in \{\frac{m_0}{q}, \frac{m_0+1}{q},
\ldots, d_{1}-\frac{1}{q}\}\\\
& =  & 
\sum_{Q\in {\tilde Q_{i+1}}\cup \cdots \cup {\tilde Q_{s_1}}}e_0(R/Q) 
+\Delta_i & 
\mbox{if}~~ \frac{\lfloor xq\rfloor}{q}\in \{d_i,d_i+\frac{1}{q}, \ldots, 
 d_{i}+\frac{m_0-1}{q})\\\
& = & 
\sum_{Q\in {\tilde Q_{i+1}}\cup \cdots \cup {\tilde Q_{s_1}} }e_0(R/Q) 
& \mbox{if}~~  \frac{\lfloor xq\rfloor}{q}\in \{d_i+\frac{m_0}{q}, 
\ldots, d_{i+1} -\frac{1}{q}\},
\end{array}$$

where  $|\Delta_i|\leq 
\sum_{Q\in {\tilde Q_1}\cup \cdots \cup {\tilde Q_i}}e_0(R/Q)$.
Therefore for $x\not\in T_{R, I} = \{0 = d_0, d_{1}, d_{2},\ldots, d_{s_1}\}$

$$\begin{array}{llll}
\lim_{n\to \infty}f_n(R, I)(x)
 & =  & \sum_{Q\in {\tilde Q_1}\cup \cdots \cup {\tilde Q_{s_1}}}
e_0(R/Q) & 
\mbox{if}~~ 0 < x < d_{1_1}\\\
 & =  & \sum_{Q\in {\tilde Q_2}\cup \cdots \cup {\tilde Q_{s_1}}}e_0(R/Q) & 
\mbox{if}~~ d_{1} < x < d_{2}\\\
& =  & \sum_{Q\in {\tilde Q_{i+1}}\cup \cdots \cup {\tilde Q_{s_1}}}e_0(R/Q) & 
\mbox{if}~~ d_{i} < x < d_{i+1}\\\
  & = & 0 & 
\mbox{if}~~ x\geq d_{s_1}\end{array}$$
Now we take
$C_{i+1} = \sum_{Q\in {\tilde Q_{i+1}}\cup \cdots \cup {\tilde Q_{s_1}} }e_0(R/Q)$.
\end{proof}

\begin{lemma}\label{l1}The map 
${\tilde \Phi}:\sC_k \longto H(\C)$ given by
$$(R,I)\to F_{R,I},~\mbox{ where}~
F_{R, I}(z) = \lim_{n\to \infty} {\widehat f_n} (R, I)(z)$$
is a well defined function.

Moreover, if $\dim R\geq 1$ then $F_{R, I}\equiv  {\widehat f}_{R, I}$, where
${f}_{R, I}$ is  the HK density function of $(R, I)$.
\end{lemma}

\vspace{5pt}

\begin{proof}~~{\underline{Case}}~(1)\quad  $\dim~R = 0$. Then, for $q = p^n$, 
$f_n(R, I)(x) = q\ell(R_m)$ for  $x\in [m/q, m+1/q)$.
If $m_0$ is an integer such that $R_m = 0$ for 
$m\geq m_0$ then  $f_n(R, I)(x) = 0$ for $x\geq m_0/q$.
Therefore 
$${\widehat f_n}(R, I)(z)  = q\int_0^{m_0/q}f_n(R,I)(t)e^{itz}dt
 = q\sum_{m=0}^{m_0-1}\ell\Big(\frac{R}{I^{[q]}}\Big)_m\int_{m/q}^{m+1/q}e^{itz}dt.$$

\vspace{5pt}
\noindent{\bf Claim}.\quad 
$\lim_{n\to \infty} q \int_{m/q}^{m+1/q}e^{itz}dt = 1$.

\vspace{5pt}
\noindent{Proof of the claim}:\quad Let $c_0 = iz$.
Now 
$$\begin{array}{lcl}q\int_{m/q}^{m+1/q}e^{c_0t}  & = & 
\frac{q}{c_0}\left[e^{c_0(m+1)/q}-e^{c_0m/q}\right]\\\
& = & \frac{q}{c_0}\left[\frac{c_0}{q}\right]+
\frac{q}{c_0}\frac{(c_0(m+1))^2}{q^2}(A_{m+1}) - 
\frac{q}{c_0}\frac{(c_0m)^2}{q^2}(A_{m})\\\
& = & 1+c_0(m+1)^2\frac{A_{m+1}}{q} -
c_0m^2\frac{A_{m+1}}{q},\end{array}$$
where $A_l= \frac{1}{2!}+\frac{c_0l/q}{3!}+
\frac{(c_0l/q)^2}{4!}+\cdots  $.
But there exists constants $C_{m+1, z}$ and  $C_{m, z}$
depending on $m$ and $z$ such that 
$$A_{m+1} \leq e^{|c_0(m+1)|/q} \leq C_{m+1, z} \quad \mbox{and}\quad
A_{m} \leq e^{|c_0m|/q} \leq C_{m, z}.$$
Hence $\lim_{n\to \infty}q\int_{m/q}^{m+1/q}e^{c_0t} = 1$.

In particular 
$$\lim_{n\to \infty} {\widehat f_n}(R, I)(z) = 
\lim_{n\to \infty}\sum_{m=0}^{m_0-1}\ell\Big(\frac{R}{I^{[q]}}\Big)_m
=
\sum_{m=0}^{m_0-1}\ell(R_m) = 
\ell(R),~~\mbox{for all}~~z\in \C$$
and therefore $F_{R, I}$ is a constant  and 
hence  an entire function.

\vspace{5pt}

\noindent{\underline{Case}}~(2)~~~$\dim~R = 1$. 
By Lemma~\ref{l11}, 
there exist constants $\alpha = C_{s_1}$ and $M$ in $\R$ such that 
$$\cup_n(\mbox{Supp}~f_n(R,I))\cup (\mbox{Supp}~f_{R,I}) \subseteq [0, \alpha]~\mbox{and}~ 
\|f_n(R,I)(t)-f_{R,I}(t)\| \leq M.$$

Fix  $z= x+iy \in \C$ and let $\epsilon >0$. Let 
$\epsilon_1 = \epsilon/(M\cdot e^{|\alpha y|})$.
By Lemma~\ref{l11} there is an open neighborhood 
$T_{\epsilon_1}$ of the set $d_0, d_{1},\ldots, d_{s_1}$ such that 
$T_{\epsilon_1}$ is  of measure 
$< \epsilon_1 $ and 
$$|f_n(R,I)(t)-f_{R,I}(t)| = 0,~\mbox{ for}~ t\in \R\setminus T_{\epsilon_1}.$$

Then 
$$|{\widehat f_n}(R,I)(z)-{\widehat f_{R,I}}(z)|\leq   \int_{T_{\epsilon_1}} 
|f_n(R,I)(t)-f_{R,I}(t)|e^{|ty|}dt < \epsilon.$$
This 
 proves the pointwise convergence of  
the sequence  
$\{{\widehat f_n}(R,I)\}_n$ to the function ${\widehat f}_{R,I}$.
Since $f_{R,I}\in L^1([0, \alpha])$ its Fourier transform  
${\widehat f}_{R,I}$
is an entire function.

\vspace{5pt}

\noindent{\underline{Case}}~(2)~~If  $\dim~R \geq 2$ then the assertion 
follows by the same argument as above.
\end{proof}

\begin{lemma}\label{l2}The map ${\tilde \Phi}:(\sC_k, \tensor)\longto H(\C)$ given by 
$(R, I) \to F_{R, I}$ is multiplicative, that is
$F_{(R, I) \tensor (S, J)} = F_{R, I}\cdot F_{S, J}$ such that 
the identity element $(k, (0))$ of $(\sC_k, \tensor)$ maps to the identity element
$F\equiv 1$ of $H(\C)$.
\end{lemma}
\begin{proof}
\noindent{\bf Case}~(1)\quad Let $\dim R = 0$ and $\dim S =0$, then 
$\dim R\tensor S =0$. By Lemma~\ref{l1}, for all $z\in \C$
$$F_{R\tensor S, I\tensor J}(z) = \ell(R\tensor S) =
\ell(R)\ell(S) = F_{R, I}(z)\cdot F_{S,J}(z).$$ 
 
\vspace{5pt}

\noindent{\bf Case}~(2)\quad Let $\dim R = 0$ and $\dim S \geq 1$. Then
 $R = R_0\oplus R_1\oplus \cdots \oplus R_{n_0}$, for some $n_0$ and 
$\dim (R\tensor S)\geq 1$. 
First we show that 
$f_{R\tensor S, I\tensor J} = \ell(R)\cdot f_{S, J}$.

 Let $T_{S, J}$ be the finite set for the pair $(S, J)$, as  in Lemma~\ref{l11}.
Let us fix  $x\in \R_{>0}\setminus T_{S, J}$. 
 We can choose $q\gg 0$ such that 
$\lfloor xq \rfloor > n_0$ and $I^{[q]} = 0$.
Further we assume that 
 the points 
$x, x-(1/q), \ldots, x-(n_0/q)$  avoid the set $T_{S, J}$. 
Then

\begin{multline*}
{f}_{R\tensor S, I\tensor J}(x) = \lim_{n\to \infty} f_n(R\tensor S, I\tensor J)(x) 
= \lim_{n\to \infty} \sum_{i=0}^{n_0}\ell(R_i)\ell(S/J^{[q]})_{{\lfloor xq \rfloor}-i}\\\
 = \sum_{i=0}^{n_0}\ell(R_i)\lim_{n\to \infty}
f_n(S, J)\big(x-\tfrac{i}{q}\big) = \ell(R)\cdot f_{S,J}(x).\end{multline*}

Applying  Fourier transform functor on both the sides we get
 $$F_{R\tensor S, I\tensor J}(z) = 
{\widehat f}_{R\tensor S, I\tensor J}(z) = \ell(R)\cdot {\widehat f}_{S, J}(z) = 
F_{R, I}(z)\cdot F_{S,J}(z).$$

\vspace{5pt}

\noindent{\bf Case}~(3)\quad Let $\dim R \geq 1$ and $\dim S \geq 1$.
For the sake of abbreviation we write $f_n(R, I)= f_n $,
$f_n(S, J) =  g_n$.

\vspace{5pt}

\noindent{\bf Claim}.\quad For every $x\in \R$

\begin{enumerate}
\item[a]
$\lim_{n\to \infty}(f_n * g_n)(x) \longto 
(f_{R, I}*f_{S, J})(x)$,  
\item[b] $\lim_{n\to \infty}f_n(R\tensor S, I\tensor J)(x) 
 \lim_{n\to \infty}(f_n * g_n)(x)$.
\end{enumerate}

If we assume  the proof of the claim for the moment then applying the 
Fourier transform we get 
$$F_{R\tensor S, I\tensor J} = {\widehat f_{R\tensor S, I\tensor J}}
= {\widehat (\lim_{n\to \infty}f_n(R\tensor S, I\tensor J))}
= {\widehat {f_{R,I} * f_{S, J}}} = {F_{R,I}}\cdot {F_{S,J}}.$$

\vspace{5pt}

\noindent{\underline{Proof of the claim}~[a]}:\quad
By  Lemma~\ref{l11},  there exist finite sets $T_{R,I}$ and 
$T_{S, J}$ (possibly empty) in $\R$ such that the sequences 
 $\{f_n\}$ and $\{g_n\}$ converge 
uniformly to $f_{R, I}$ and $f_{S, J}$ respectively outside a set 
of arbitrarily small measure containing $T_{R, I}$ and $T_{S, J}$.

Moreover $f_{R, I}$ is a continuous function outside a set of 
measure $0$ (hence so is the function $t\to f_{R, I}(x-t)$ for a fixed $x$).
Let 
$$M = \mbox{max}~\{||f_n||, ||g_n||, ||f_{R, I}||, ||f_{S,J}||\mid n\in \N\}.$$
Fix an $x\in \R$. Let $\epsilon > 0$ be any number.
We choose a  set $U$ of the measure $\leq \epsilon/4M^2$ such 
that $\{f_n\}$ and $\{g_n\}$ converge outside $U$.
Further we choose $n_0\geq 0$ such that  for $n\geq n_0$ 
$$|g_n(t)-f_{S, J}(t)| <\epsilon/4|x|M\quad \mbox{and}\quad 
|f_n(x-t)-f_{R, I}(x-t)| <\epsilon/4|x|M
\quad \mbox{for}\quad  t\in \R\setminus U.$$
Now 
\begin{multline*}|(f_n * g_n)(x) - (f_{R, I} * g_{S, J})(x)| \leq   
\int_0^x |f_n(x-t)||g_n(t)-f_{S, J}(t)|dt\\\
 + \int_0^{x} |f_n(x-t)|-f_{R, I}(x-t)||f_{S, J}(t)|dt
+\int_U|f_n(x-t)g_n(t)|dt +  \int_U|f_{R, I}(x-t)f_{S, J}(t)|dt\\\
< |x|\cdot M \cdot \frac{\epsilon}{4|x|M} +  
|x| \cdot \frac{\epsilon}{4|x|M} \cdot M + M^2\cdot \frac{\epsilon}{4M^2} + 
M^2\cdot \frac{\epsilon}{4M^2} = \epsilon. \end{multline*}

\vspace{5pt}

\noindent{\underline{Proof of the claim}~[b]}:\quad
Now $\dim(R\tensor S)\geq 2$ implies that $\{f_n(R\tensor S, I\tensor J)\}_n$
converges to   $f_{R\tensor S, I\tensor J}$ uniformly, where 
$$f_n(R\tensor S, I\tensor J) = 
 = \frac{1}{q^{d_1+d_2-1}}
\Big(\sum_{l=0}^{\lfloor xq \rfloor}\ell\big(\frac{R}{I^{[q]}}\big)_l\cdot 
\ell\big(\frac{S}{J^{[q]}}\big)_{\lfloor xq \rfloor-l}\Big)\\\\\
  = \frac{1}{q}\sum_{l=0}^{\lfloor xq \rfloor}
f_n\big(\tfrac{l}{q}\big)g_n\big(\tfrac{\lfloor xq \rfloor-l}{q}\big).$$

On the other hand, let $T_{S, J} = \{d_0, \ldots, d_{s_1}\}
$ be the finite set as in 
 Lemma~\ref{l11}. Let 

$$M_{x,q} = \big\{l\in \Z\mid  \tfrac{\lfloor xq \rfloor}{q} -\tfrac{l+1}{q}
\quad{and}\quad 
 \tfrac{\lfloor xq \rfloor}{q} -\tfrac{l}{q}\not\in
\{d_i, \ldots d_i+\tfrac{m_0+1}{q}\},\quad\mbox{for any common}\quad d_i\big\}.$$
It is easy to check that the cardinality of $M_{x, q} \leq (m_0+1)d_{s_1}+1$,
and for any $l\not\in M_{x,q}$  
$$g_n\big({\lfloor xq \rfloor}/{q} -({l+1})/{q}\big) = 
g_n\big({\lfloor xq \rfloor}/{q} -{l}/{q}\big).$$
We can write 
\begin{multline*}
(f_n\ast g_n)(x) = \int_0^xf_n(t)g_n(x-t)dt = 
\sum_{\{0\leq l \leq \tfrac{\lfloor xq \rfloor-1}{q}\mid l\not\in M_{xq}\}}
\int_{l/q}^{(l+1)/q}f_n(t)g_n(x-t)dt\\\ 
+ \int_{\frac{\lfloor xq \rfloor}{q}}^x f_n(t)g_n(x-t)dt + 
\sum_{l\in M_{x,q}}\int_{l/q}^{(l+1)/q}f_n(t)g_(x-t)dt. \end{multline*}
But 
$$\begin{array}{lcl}
\int_{l/q}^{(l+1)/q}f_n(t)g_n(x-t)dt & = 
 & f_n\big(\tfrac{l}{q}\big)
g_n\big(\tfrac{\lfloor xq \rfloor-l-1}{q}\big) = f_n\big(\tfrac{l}{q}\big)
g_n\big(\tfrac{\lfloor xq \rfloor-l}{q}\big)\quad 
\mbox{if}\;\;l\not\in M_{x, q}\\\\
 & \leq & M^2/q\quad \mbox{for any}\;l\end{array}$$
Hence, 
$$|f_n(R\tensor S, I\tensor J)(x) - 
\big(f_n\ast g_n)(x)|\leq \big((m_0+1)d_{s_1}+2\big)\cdot \frac{M^2}{q}.$$
This proves the claim~[b] and hence the lemma.
\end{proof}

\begin{rmk}
The case when both  $(R, I)$ and $(S, J)$ are of dimension $\geq 2$ the 
above lemma can also be deduced from  the thesis of M. Mondal ([MM]) and Lemma~\ref{l1}.
\end{rmk}

\begin{rmk}\label{r1} \begin{enumerate}
\item By Lemma~\ref{l2} the map  
${\tilde \Phi}:\sC_k \longto H(\C)$ is a multiplicative map of monoids 
which takes  the identity $(k, (0))$ of the monoid $\sC_k$ to the multiplicative 
identity of $H(\C)$. 
 If  $F:L^1(\R)\longto H(\C)$  denotes the Fourier transform functor
then  $Im(F)\cap Im({\tilde \Phi})$ is closed under multiplication 
and contains whole of $Im({\tilde \Phi})$ except ${\tilde \Phi}(R, I)$, 
where $\dim~R = 0$.

This is because if $f\in L^1(\R)$ 
then the Fourier transform ${\widehat f}\mid_{\R}(x)$   
 tends to $0$ as $x\to \infty$ (Theorem~9.6 in [R]), whereas 
${\tilde \Phi}(R, I) = F_{R, I}$ is a nonzero
constant function everywhere. However 
${\tilde \Phi}(R, I) = \lim_{n\to \infty}{\widehat f_n(R, I)}$, where 
 ${\widehat f_n}(R, I)\in \mbox{Im} (F)$ and $f_n(R, I)$ are the Dirac 
functions.

This is 
analogous to the case where the  set $C_c(\R)$ of 
compactly supported continuous 
function on $\R$, by definition is the intersection of the set of compactly supported 
function and  the set of continuous functions on $\R$.
This set is  
closed under the convolution operation. But the identity 
element or any nonzero constant maps  is the limit of  Dirac
functions in $C_c(\R)$ does not belong to $C_c(\R)$.
\end{enumerate}
\end{rmk}

For given SG pair $(R,I)$, the function $F_{R, I}$ is 
{\em additive} for maximal dimension  components of $\Spec~R$:
 if $\Lambda = \{p\in \Spec~R\mid \dim R/p = \dim R\}$ then 
$$F_{R, I} = \sum_{p\in \Lambda} \lambda(R_p)F_{R/p, I+p/p},$$
which follows  by Lemma~\ref{l1} (in case $\dim R = 1$) 
and by Proposition~2.14 of [T] (in case $\dim R \geq 2$).

However the formulation ignors lower dimensional components of $R$.
Here   we   extend the definition of
${\tilde \Phi}:(\sC_k, \tensor)\longto H(\C)$ 
to a multiplicative map of monoids
 $\sC_k \longto H(\C)[X]$ which keeps track of every 
irreducible components of $\Spec R$.

\begin{notations} For a SG pair $(R,I)$ of dimension $d$, we write  
$$\{\mbox{minimal primes of $R$}\} = {\tilde P_1} \cup {\tilde P_2} \cup 
\cdots{\tilde P_d},$$
where ${\tilde P_i} = 
\{{\bf p}_{i1}, \ldots, {\bf p}_{ij_i}\in \Spec R\mid \dim~R/{\bf p}_{ij} = i\}.$
 Let $Q_{ij}$  denote the ${\bf p}_{ij}$-primary component of $R$. Then 
$R^i = R/(Q_{i1}\cap \cdots \cap Q_{ij_i})$ is the union of $i$-dimensional 
components of $R$. 

By Lemma~\ref{l11}
$$F_{R^i, IR^i} = \sum_{j=1}^{j_i} \ell(R_{{\bf p}_{ij}})
F_{R/{\bf p}_{ij}, I+{\bf p}_{ij}/{\bf p}_{ij}}.$$
\end{notations}

\begin{propose}\label{p1}The map $\Pi: \sC_k \longto H(\C)[X]$ given by 
$$(R, I) \to F_{R^0, IR^0} + F_{R^1, IR^1}X + \cdots + F_{R^d, IR^d}X^d$$
 is multiplicative.
\end{propose}
\begin{proof}First we prove the following 

\vspace{5pt}

\noindent{\bf Claim}.\quad Let $(R, I)$ and $(S, J)$ be in $\sC_k$.
Then 
$$\mbox{min prime}(R\tensor S) = \{{\bf p}\tensor S+
R\tensor {\bf q}\mid {\bf p}\in \mbox{min prime}(R),~{\bf q}\in 
\mbox{min prime}(S)\}.$$ 
Further if  ${\bf p}\tensor S + R\tensor {\bf q}\in 
\mbox{min prime}~(R\tensor S)$ then  
$\ell(R_{\bf p})\ell(S_{\bf q}) = \ell((R\tensor S)_P)$, where 
$P = {\bf p}\tensor S + R\tensor {\bf q}$.

\vspace{5pt}

\noindent{\underline{Proof of the claim}}:\quad 
Let  ${\bf p}\in \Spec R$ and
${\bf q}\in \Spec S$.
Since $k$ is algebraically closed (see Ex. 3.15 of [H])

$$\frac{R\tensor S}{{\bf p}\tensor S+R\tensor {\bf q}} \simeq 
{R}/{{\bf p}}\tensor_k{S}/{{\bf q}},$$
is an integral  domain, and therefore
${\bf p}\tensor S+R\tensor {\bf q}$ is a prime ideal of $R\tensor S$.
 
On the other hand,
if $P\in \Spec R\tensor S$ then 
$\phi_R:R\to (R\tensor S)/P$ given by $r\to r\tensor 1$
and $\phi_S:S\to (R\tensor S)/P$ given by $s\to 1\tensor s$ are ring homomorphisms. 
Therefore 
${\bf p} = \ker(\phi_R)$ 
and ${\bf q} = \ker(\phi_S)$ are the prime ideals of $R$ and $S$ respectively such 
that 
 $P\supseteq {\bf p}\tensor S + R\tensor {\bf q}$.
In particular 
$$\mbox{min prime} (R\tensor S) \subseteq \{{\bf p}\tensor S+R\tensor {\bf q}\mid {\bf p}\in 
\Spec R,~~{\bf q}\in \Spec S\}\subseteq \Spec (R\tensor S)$$
which implies the first assertion of the claim.

Now consider  $P = {\bf p}\tensor R+ R\tensor {\bf q} \in \mbox{min prime}~(R\tensor S)$.
Then  $r\in R\setminus {\bf p}$ and $s\in S\setminus  {\bf q}$ 
implies $r\tensor 1$ and 
$1\tensor s$ are in $(R\tensor S)\setminus P$.

So if $T_1 = \{r\tensor s\mid r\in R\setminus {\bf p}~\mbox{and}~ s\in S
\setminus {\bf q}\}$ 
and $T= (R\tensor S)\setminus P$ then $T_1^{-1}\subseteq T$  
$$R_{\bf p}\tensor S_{\bf q} = T_1^{-1}(R\tensor S)\quad\mbox{and}\quad
  T^{-1}(R\tensor S) = (R\tensor S)_P.$$
Similarly, if $k({\bf p})$, $k({\bf q})$ and  $k(P)$ are
 the residue fields of ${\bf p}$, 
${\bf q}$ and $P$ respectively
then 
$$ k({\bf p})\tensor_k k({\bf q}) = T_1^{-1}\Big(\frac{R}{{\bf p}}\tensor_k 
\frac{S}{{\bf q}}\Big) = T_1^{-1}\Big(\frac{R\tensor_kS}{P}\Big)\quad\mbox{and}\quad
T^{-1}\Big(\frac{R\tensor_kS}{P}\Big) = k(P).$$

Hence $k(P) = T^{-1}(k({\bf p})\tensor_k k({\bf q}))$.
By  Cohen structure theorem, $k({\bf p})\subset R_{\bf p}$, $k({\bf q})\subset
S_{\bf q}$ and $k(P)\subset (R\tensor S)$.
Now if  
$\ell_{k({\bf p})}(R_{\bf p}) = m_1$,
and $\ell_{k({\bf q})}(S_{\bf q}) = m_2$
then 
we have 
$$0 = M_0 \subset M_1\subset M_2 \subset \cdots \subset M_{m_1m_2} =
R_{\bf p}\tensor_k S_{\bf q},$$
where $M_i/M_{i-1} \simeq k({\bf p})\tensor_kk({\bf q})$. But then 
$$0 = T^{-1}M_0 \subset T^{-1}M_1\subset T^{-1}M_2 \subset \cdots \subset 
T^{-1}M_{m_1m_2} =
T^{-1}(R_{\bf p}\tensor_k S_{\bf q}) = (R\tensor_kS)_P,$$
where $T^{-1}(M_i/M_{i-1}) = k(P)$.
This proves the second part of the claim.

\vspace{5pt}

Let $(R, I)$ and $(S, J)$ be two SG pairs in $\sC_k$.
Let 
$${\tilde P_l} = 
\{P_{lj}\in \mbox{min prime}(R\tensor_k S)\mid \dim(R\tensor S)/P_{lj} = l\}.$$
Similarly we can define the subsets 
${\tilde {\bf p}}_{l_1}\subset \mbox{min prime}(R)$ and
${\tilde {\bf q}}_{l_2}\subset \mbox{min prime}(S)$.
We have canonical bijection 
$${\tilde P_l} \leftrightarrow {\tilde {\bf p}}_{l_1}\times
{\tilde {\bf q}}_{l_2}\quad  \mbox{where}\quad 
P_{lj}\to ({\bf p}_{l_1j}, {\bf q}_{l_2j})$$
 if $P_{lj} = {\bf p}_{l_1j}\tensor S+R\tensor {\bf q}_{l_2j}$,  It is obvious that 
$l_1+l_2 = l$ as $\dim R/{\bf p}_{lj} + \dim S/{\bf q}_{lj} = l$.
By Lemma~\ref{l1}, 
$$\Phi_{\frac{R\tensor S}{P_{lj}}, \frac{I\tensor J+P_{lj}}{P_{lj}}} = 
\Phi_{\frac{R}{{\bf p}_{lj}}, \frac{I+{\bf p}_{lj}}{{\bf p}_{lj}}}\cdot 
\Phi_{\frac{S}{{\bf q}_{lj}}, \frac{I+{\bf q}_{lj}}{{\bf q}_{lj}}}.$$
Therefore 
$$\Phi_{(R\tensor S)^l, (I\tensor J)(R\tensor S)^l} =
\sum_{P_{lj}\in {\tilde P_l}}\ell(R\tensor S)_{P_{lj}}
\Phi_{\frac{R\tensor S}{P_{lj}}, \frac{I\tensor J+P_{lj}}{P_{lj}}}
 = \sum_{l_1+l_2 =l}\Phi_{R^{l_1}, IR^{l_1}}\cdot \Phi_{S^{l_2}, JS^{l_2}} .$$

Hence $\Pi(R\tensor S, I\tensor J) = \Pi(R, I)\cdot \Pi(S,J)$.

\end{proof}

 We proceed to extend the monoid $(\sC_k, \tensor)$ to a semi ring 
$(\{\mbox{graded pairs}\}/\equiv, \tensor, {\tilde \oplus})$ as follows:
by a graded pair $(R,I)$ we mean 
$(R,I)$ is a  finite sum of SG pairs,
where for two SG pairs 
$(R,I)$ and $(S,J)\in \sC_k$ we define the addition of pairs as
(where  $(R\oplus S)_n = R_n\oplus S_n$ and $(I\oplus J)_n = I_n\oplus J_n$) 
$$(R,I){\tilde \oplus}(S,J) = (R\oplus S, I\oplus J),$$
with addition and multiplication as  the pointwise  addition and the pointwise 
multiplication resply. In other words we are considering the direct product 
of rings. We canonically extend this addition operation to finitely many
sg pairs, and multiplication $\tensor$ as before.
Also
 $(R, I) \equiv (S,J)$ if there is a  degree $0$ graded isomorphism 
$\eta:R\longrightarrow
S$ of rings such that $\eta(I) = J$.

Following lemma gives the 
 uniqueness of the decomposition of a graded pairs into the SG pairs.

\begin{lemma}\label{l0}Let $(R, I)$ and $(S, J)$ be graded pairs such that 
$$(R, I) = {\tilde \oplus}_i {\tilde \oplus}^{m_i}(R^i, I^i)~\mbox{ and}~ 
(S, J) = {\tilde \oplus}_j{\tilde \oplus}^{n_j} (S^j, J^j),$$
 where $\{(R^i, I^i)\}_i$ and $\{(S^j,J^j)\}_j$ are 
two  finite sets of 
distinct SG 
pairs. If  there is an isomorphism of graded pairs
$\eta:(R, I)\simeq (S, J)$ then
for every $i$ there is $j_i$ such that $(R^i,I^i)\simeq (S^j,J^j)$ 
and $m_i = n_{j_i}$.

In particular, the decomposition of a graded pair into a finite sum of  sg pairs 
is unique up to an isomorphism of graded pairs.
\end{lemma}

\begin{proof}The map $\eta$  gives  the isomorphism of rings 
$\eta:R\longrightarrow S$, which 
 induces the  homeomorphism
$\eta^*:\oplus_j\oplus^{n_j}\Spec S^j\longto \oplus_i\oplus^{m_i}\Spec R^i$.
Note that each $\Spec~R^i$ and $\Spec~S^j$ is a  connected set (being spectrum of a 
standard graded ring) and hence, if $m= \sum_im_i$ and $n= \sum_j n_j$ then 
$m=n$.

We rewrite $R = \oplus_{i=1}^m R^i$ and $S = \oplus_{j=1}^mS^j$, where 
$\{R^i\}$ and $\{S^j\}$ need not be sets of distinct elements.
Here we identify $R^i$ as $(0, \ldots, R^i, \ldots, 0)\subset {\tilde R}$ 
(similarly for $S^j$).

Now we prove that  the map 
$\eta\mid_{R^1}:R^1 \simeq S^j$ is an 
isomorphism, for some $j$.

Note that 
an idempotent of $R$ is the sum of the idempotents of $R^i$.
Hence   the ring $R$ has precisely $m$
 nontrivial irreducible idempotents ({\em i.e.}, the idempotent which can not 
be written as a sum of two nontrivial idempotents) namely
$$\{e_i := (a_1,\ldots, a_m)\mid 
  a_i =1~~\mbox{and}~~ a_j= 0~~\mbox{for}~ j\neq i\}.$$
Similarly $S$ has 
precisely $m$
 nontrivial irreducible idempotents say $\{f_i\}_i$.
Since the (irreducible) idempotents map to the (irreducible) idempotents, for $e_1$ 
there is $f_j$ such that $\eta(e_1) = f_j$, for some $j$.    
This gives $\eta(R_1) = \eta(Re_1) = Sf_j = S^j$. Therefore the 
induced map $\eta:R^1\longto S^j$ is an isomorphism.

Also $\eta(I^1) \subseteq S^j\cap (J^1\oplus \cdots \oplus J^m) = J^j$.
Therefore $\eta(R^1, I^1) \simeq (S^j, J^j)$. 
This gives an isomorphism 
$$\eta:{\tilde \oplus}_{ \{2\leq i\leq m\}}(R^i, I^i) 
\longrightarrow {\tilde \oplus}_{\{1\leq i \leq m\mid i\neq j\}}(S^j, J^j).$$ 
Now the proof follows by induction on $m$.
\end{proof}

Note that the  monoid $(\sC_k, \tensor)$ also gives the corresponding 
commutative semi ring
 $(\N[\sC_k], *, +)$, where  as a set  
 $$\N[\sC_k] = \{\sum_{P_m\in \sC_k}r_mP_m \mid
r_m\in \N~~\mbox{are all zero except for finitely many},~~P_m \in \sC_k \},$$
and where the addition $+$ is a formal sum of elements of $\sC_k$ given by 
$$\Big(\sum_{P_m\in \sC_k}r_m P_m\Big)  + \Big(\sum_{P_m\in 
\sC_k}r'_m P_m\Big) = \sum_{P_m\in\sC_k}\Big(r_m+r'_m\Big) P_m$$
 and the multiplication is 
$$\Big(\sum_{P_m\in\sC_k}r_m P_m \Big) *
\Big(\sum_{P_m\in \sC_k}r'_m
 P_m\Big) = \sum_{P_{\tilde m} \in \sC_k}
\Big(\sum_{P_m P_n = P_{\tilde m}}r_mr'_n\Big)
 P_{\tilde m}.$$

By the uniquness of the decomposition of graded pairs 
(Lemma~\ref{l0})
there is an isomorphism
$$(\N[\sC_k], *, +) \simeq  (\{\mbox{graded pairs}\}/\equiv, \tensor, {\tilde \oplus})$$
of semi rings given by 
$\sum_i m_i(A_i,I_i) \to {\tilde \oplus}_i{\tilde \oplus}^{m_i} (A_i,I_i)$.
 Moreover 
the addition 
satisfies the cancellation law ({\em i.e.}, $a+c = b+c$
implies $a=b$ hence   (Theorem~20.8 of [W]) the semi ring  $\N[\sC_k]$  
embeds into the ring  
$\Z[\sC_k] = \{a-b\mid a, b\in \N[\sC_k|\}$.
In particular we have the embedding (respecting the  binary operations)
$$(\sC_k, \tensor) \hookrightarrow (\N[\sC_k], \tensor, {\tilde \oplus}) \simeq
(\frac{\mbox{\{graded pairs}\}}{\equiv}, \tensor, {\tilde \oplus})
\hookrightarrow (\Z[\sC_k], \tensor, {\tilde \oplus}),$$
where the first 
embedding is given by $(R,I) \to 1\cdot (R,I)$. 
Here the SG pairs $(k, (0))$ and $(0,(0))$ are respectively  the multiplicative  
and  the additive identity  of the ring $\Z[\sC_k]$. 

In particular, if  $\phi:\sC_k\longto (S, +, \cdot)$ is a map where 
$(S, +, \cdot)$ is a ring 
and where  
$$\phi((R,I) \tensor (R', I')) = \phi(R,I) \cdot \phi(R', I')\quad\mbox{and}\quad
\phi((R,I) {\tilde \oplus} (R', I')) = \phi(R,I) + \phi(R', I'),$$ 
then $\phi$ extends 
uniquely to a map of rings
$\phi:(\Z[\sC_k], \tensor, {\tilde \oplus})\longto (S, +, \cdot)$.

\begin{thm}\label{t1}The map $\Pi: \Z[\sC_k] \longto H(\C)[X]$ given by 
$$(R, I) \to F_{R^0, IR^0} + F_{R^1, IR^1}X + \cdots + F_{R^d, IR^d}X^d$$
 is a ring homomorphim, where $d = \dim~R$.
\end{thm}
\begin{proof}We only need to check that, if $(R, I)$ and $(S, J)$ are two SG 
pairs then 
\begin{equation}\label{add}F_{(R\oplus S)^i, (I\oplus J)^i} = 
F_{R^i, IR^i}+ F_{S^i, JS^i}.\end{equation}
It is easy to check that 
the set $\{P\in \Spec (R\oplus S)\mid \dim(R\oplus S/P) = i\}$ is  equal to 
$$\{p\oplus S\mid p\in \Spec R\mid \dim(R/p) = i\}\cup
\{R\oplus q\mid q\in \Spec S\mid \dim(S/q) = i\}.$$
Moreover $(R\oplus S)_P = R_p$, if $P=p\oplus S$.
Now (\ref{add}) follows from the additivity property of $\Phi$.
\end{proof}

\begin{cor}\label{c1}The map $\Pi_e:\Z[\sC_k] \longto \R[X]$
 given by 
 $$(R,I) \to e_{HK}(R^0, IR^0)+ e_{HK}(R^1, IR^1)X+\cdots + 
e_{HK}(R^d, IR^d)X^d$$
is a ring homomorphism.

Also the map $\Pi_p:\sC_k\longto \R$ given by 
$(R,I) \longto e_{HK}(R,I)$ is multiplicative, {\em i.e.},
$e_{HK}(R\tensor S, I\tensor J) = e_{HK}(R, I)e_{HK}(S, J)$.
 \end{cor}
\begin{proof}Consider the evaluation map 
  $ev:H(\C)[X]\longto \C[X]$
given by 
$\sum_i f_iX^i\to \sum_if_i(0)X^i$. Since this is a ring homomorphism  the map 
$\Pi_e := ev\circ\Pi$  is a ring homomorphism.

Moreover the map  $pr :\C[X]\longto \C$ given by 
$a_0+a_1X+\cdots +a_mX^m\to a_m$ is multiplicative and hence so is
 $\Pi_p = pr\circ \Pi\mid_{\sC_k}$. Now the corollary follows as
$e_{HK}(R, I)$ by definition is $e_{HK}(R^d, IR^d)$, where $d$ is dimension of $R$.
\end{proof}

\begin{rmk}\label{d2}
For two SG pairs $(R, I)$ and $(S, J)\in \sC_k$
we define $(R,I)\bigcup (S,J)\in \sC_k$ as a standard graded ring given by 
$(R\bigcup S)_0 = k$ and $(R\bigcup S)_n = R_n\oplus S_n$.

Let  $\sC_k^{1} = \{(R, I)\in \sC_k\mid \dim~R\geq 1\} $
then $\sC_k^{1}$ is  a $\Z[\sC_k]$-module and 
it is easy to check that for  $(R, I)$  and $(S, J)$ in $\sC{1}$
 $$\Pi((R, I){\tilde \oplus}(S, J))  = \Pi((R, I)\bigcup(S, J))$$ 

and 
 the map $\Pi: \Z[\sC_k] \longto H(\C)[X]$
factors through the quotient ring 
$$\Z[\sC_k]/
\langle \{(R, I){\tilde \oplus}(S, J) - (R, I)\bigcup(S, J)
\mid (R, I),~~(S, J)\in \sC_k{1}\}\rangle.$$
Hence 
$\Pi\mid_{\sC_k^1}:\sC_k^{1}\longrightarrow H(\C)[X]$ given by  
$$(R, I) \to {\widehat f}_{R^1, IR^1}X + \cdots + 
{\widehat f}_{R^d, IR^d}X^d$$
is  a $\Z[\sC_k]$-linear map and therefore 
$(e\circ \Pi)\mid_{\sC_k^1}:\sC_k^1 \longrightarrow \R[X]$ is a $\Z[\sC_k]$-linear
 which sends
$$(R,I) \to e_{HK}(R^1, IR^1)X+\cdots +
e_{HK}(R^d, IR^d)X^d.$$
\end{rmk}

\begin{notations} Let $(R,I)$ be a SG pair of dimension $d$. 
It is easy to check that  $e_{HK}(R,I) = {F_{R,I}}(0)$.
We denote 
$$\alpha(R,I) =  \max\{x\mid f_{R,I}(x)\neq 0\}\;\;\mbox{where}\;\;
f_{R, I} = \lim_{n\to \infty}f_n(R,I).$$
It follows that  $\dim~R = 0$ if and only if  $\alpha(R, I) = 0$.

We note  that for the Fourier transform
${\widehat f_{R, I}}$ of $f_{R, I}$, 
we have  $F_{R, I} =  {\widehat f_{R,I}}$ if  $d\geq 1$ and 
 $F_{R, I}  \neq {\widehat f}_{R,I}$, if $\dim R =0$.
\end{notations}

\vspace{5pt}

\noindent{\underline{Proof of Theorem}~\ref{t2}.}\quad 
If $(R, I)$ is a $0$ dimensional pair then 
for all $z$, 
$F_{R, I}(z) = \ell(R) = e_{HK}(R)$. Hence 
the theorem holds in this case.

Henceforth we assume $\dim~R\geq 1$. 
We denote $\alpha(R, I)$ by $\alpha$. 
Now $F_{R, I} =  {\widehat f}_{R, I}$, where $f_{R, I}\in C_c^{ae}(\R)$ is 
 the HK density function of 
$(R, I)$.

Let $A\in \R_{+}$ such that 
 ${\widehat f}_{R,I}\in PW_A$. Then there is 
a real valued function $g\in L^2[-A, A]$ such that 
$${\widehat f}_{R,I}(z) = \int_{-A}^Ag(t)e^{itz}dt =:{\widehat g}(z),
\quad\mbox{for all}\quad z\in \C.$$
However the mapping $L^2(\R)\longto L^2(\R)$ given by 
$f\to {\widehat f}$
 is an isomorphism of Hilbert spaces.
Hence $${\widehat f_{R,I}} = {\widehat g}\implies 
||{\widehat f_{R,I}} - {\widehat g}||_2 = 0\implies 
||f_{R,I}-g||_2 = \int_{\infty}^\infty |f(t)-g(t)|^2dt = 0$$

Hence $f_{R,I} = g$ a.e.. But then 
$$e_{HK}(R,I) = {\widehat f_{R,I}}(0) = 
\int_{-A}^Ag(t)dt = \int_{-A}^Af_{R,I}(t)dt = 
\int_{0}^Af_{R, I}(t)dt.$$
If $A <\alpha$ then we have $\int_A^\alpha f_{R,I}(t)dt = 0$, which is 
 a contradiction as, by Lemma~\ref{l11} and Theorem~1.1 of [T], the function 
$f_{R, I}$ is a nonnegative function which is stricly positive
 in a nbhd of $\alpha$.

For $z = x+iy\in \C$, 
$$|{\widehat f}_{R,I}(z)|\leq \int_0^\alpha |f_{R,I}(t)e^{itx -ty}|dt
\leq e^{\alpha |z|}\int_0^\alpha f_{R,I}(t)dt = e_{HK}(R,I)e^{\alpha |z|}.$$ 
Let  $C$ and $M$ be constants such that 
$|{\widehat f}_{R,I}(z)|\leq Ce^{M|z|}$.
Since ${\widehat f}_{R, I}\mid_{\R}\in L^2(\R)$, by Paley-Wiener Theorem we have 
${\widehat f_{R,I}}\in PW_M$ and hence $\alpha \leq M$.
Moreover  $e_{HK}(R,I) = {\widehat f_{R,I}}(0)
\leq C$. 
 $\Box$

\vspace{5pt}

\noindent{\underline{Proof of Corollary}~\ref{c2}.}\quad 
By  Theorem~4.9 of [TrW] and Theorem~C of [T2], we have $\alpha(R, I) = c^I({\bf m})$.
Hence the corollary follows from Theorem~\ref{t2}.$\Box$

\end{document}

\bibitem[BMS1]{BMS1}{Blickle, M., Musta\c{t}\u{a}, M., Smith, K.},
 {\it F-thresholds of hypersurfaces},
 Trans. Amer. Math. Soc. 361 (2009), no. 12, 6549-6565.
\bibitem[BMS2]{BMS2}{Blickle, M., Musta\c{t}\u{a}, M., Smith, K.},{\it
Discreteness and rationality of F-thresholds},
Michigan Math. J., 57 (2008), pp. 43-61 (Special volume in honor of
Melvin Hochster).

\bibitem[BS]{BS}{Bhargav, B., Singh, A.}, {\it The $F$-pure 
threshold of a Calabi-Yau hypersurface}, Math. Ann. (2015) 362, 551-567.

\bibitem[BSTZ]{BSTZ}{Blickle, M., Schwede, K., Takagi, S., Zhang, W.},
{\it Discreteness and rationality of F-jumping numbers on singular varieties}.
Math. Ann. 347 (2010), no. 4, 917-949.
 
\bibitem[CHSW]{CHSW}{Canton, E., Hernández, D., Schwede, K., Witt,E.}, {\it On 
the behavior of singularities at the F-pure threshold}, 
Illinois J. Math. 60 (2016), no. 3-4, 669-685. 

\bibitem[G]{G}{Gieseker, D.}, {\it Stable vector bundles and the Frobenius 
morphism}, Ann. Sci. École Norm. Sup. (4) 6 (1973), 95-101. 

\bibitem[GrS]{GrS}{Graf, P.,  Schwede, K.}, {\it Discreteness of $F$-jumping 
numbers at isolated non-$\Q$-Gorenstein points}. Proc. Amer. Math. Soc. 146 
(2018), no. 2, 473-487.

\bibitem[HY]{HY}{Hara, N., Yoshida, K.},
{\it A generalization of tight closure and multiplier ideals} Trans. Am. Math. Soc.
355, 3143-3174 (2003).

\bibitem[H]{H}{Hartshorne, R.}, {\it Algebraic geoemetry}, Springer-Verlag NY (1977).

\bibitem[HMNb]{HMNb}{Huneke, C., Montanera, J.A.,  N\'{u}\~{n}ez-Betancourt, L.},
{\it D-modules, Bernstein–Sato polynomials and F-invariants of direct summands},
Advances in Mathematics
Volume 321, 1 December 2017, Pages 298-325.

\bibitem[KSSZ]{KSSZ}{Katzman, M., Schewde, K., Singh, A., Zhang, W.},
{\it  Rings of Frobenius 
operators}. Math. Proc. Cambridge Philos. Soc. 157 (2014), no. 1, 151–167. 

\bibitem[KLZ]{KLZ}{Katzman, M., Lyubeznik, G., Zhang, W.},
{\it  On the discreteness and rationality of F-jumping coefficients},
J. Algebra 322 (2009), no. 9, 3238-3247.

\bibitem[Ma]{Ma}{Maruyama, M.}, {\it Openness of a family of torsion
 free sheaves}, J. Math. Kyoto Univ. 16-3 (1976), 627-637.

\bibitem[Mu1]{Mu1}{Mumford, D.}, {\it Abelian varieties}, Tata Institute 
of Fundamental Research, Studies in Mathematics, No~5, corrected reprint, 
Hindustan Book Agency, New Delhi (2012).

\bibitem[Mu2]{Mu2}{Mumford, D.},{\it Lectures on curves on an Algebraic Surface},
Annals of Math. studies 59, Princeton University Press, Princeton, NJ (1966).

\bibitem[MTW]{MTW}{Musta\c{t}\u{a}, M., Takagi, S., Watanabe, K.I.}, 
{\it F-thresholds and Bernstein-Sato polynomials}, European
congress of mathematics, 341-364, Eur. Math. Soc., Zurich, 2005.

\bibitem[ST]{ST}{Schwede, K., Tucker, T.}, {\it Test ideals of non-principal 
ideals: Computations, Jumping Numbers, Alterations and Division Theorems}, 
J. Math., Pures Appl. (9) 102 (2014), no. 5, 891-929.

\bibitem[TaW]{TaW}{Takagi, S., Watanabe, K.I.}, {\it On F-pure thresholds}, 
J. Algebra 282 (2004), 278-297.

\bibitem[T]{T}{Trivedi, V.}, {\it Hilbert-Kunz multiplicity and
reduction mod $p$}, Nagoya Math. Journal !85 (2007), 123-141.

\bibitem[TrW]{TrW}{Trivedi, V., Watanabe, K.}, {\it Hilbert-Kunz density functions and 
$F$-threshold}, arXiv:1808.04093v1.

\end{document}